\documentclass[12pt]{amsart}
\usepackage[latin1]{inputenc}
\usepackage{amssymb, amsmath, amscd, mathrsfs,marvosym, psfrag} 

\input xy
\xyoption{all}
\DeclareMathAlphabet{\mathpzc}{OT1}{pzc}{m}{it}

\newtheorem{theorem}{Theorem}[section]

\newtheorem{proposition}[theorem]{Proposition}

\newtheorem{lemma}[theorem]{Lemma}

\theoremstyle{definition}
\newtheorem{definition}[theorem]{Definition}

\theoremstyle{remark}
\newtheorem{remark}[theorem]{Remark}

\def\varle{\leqslant}

\newcommand{\CI}{{\mathcal I}}
\newcommand{\CJ}{{\mathcal J}}

\newcommand{\CM}{{\mathcal M}}

\newcommand{\CO}{{\mathcal O}}

\newcommand{\CW}{{\mathcal W}}

\newcommand{\CZ}{{\mathcal Z}}

\newcommand{\fh}{{{\mathfrak h}}}

\newcommand{\fg}{{{\mathfrak g}}} 
\newcommand{\fb}{{{\mathfrak b}}}

\newcommand{\fhd}{\fh^\star}

\newcommand{\hCW}{{\widehat\CW}}

\newcommand{\hCO}{{\widehat\CO}}

\newcommand{\hfh}{{\widehat\fh}}
\newcommand{\hfg}{{\widehat\fg}}
\newcommand{\hfb}{{\widehat\fb}}

\newcommand{\hR}{{\widehat R}}

\newcommand{\hfhd}{\widehat{\fh}^\star}

\newcommand{\DC}{{\mathbb C}}

\newcommand{\DZ}{{\mathbb Z}}

\newcommand{\DN}{{\mathbb N}}

\newcommand{\End}{{\operatorname{End}}}
\newcommand{\Ext}{{\operatorname{Ext}}}
\newcommand{\Hom}{{\operatorname{Hom}}}

\DeclareMathOperator{\cha}{\mathrm{ch}}

\newcommand{\rk}{{{\operatorname{rk}}}}

\newcommand{\ol}{\overline}

\newcommand{\id}{{\operatorname{id}}}
\newcommand{\res}{{\operatorname{res}}}
\newcommand{\re}{{\operatorname{re}}}
\newcommand{\crit}{{\operatorname{crit}}}

\newcommand{\inj}{{\hookrightarrow}}
\newcommand{\sur}{\mbox{$\to\!\!\!\!\!\to$}}

\newcommand{\GL}{{\operatorname{GL}}}

\newcommand{\rCO}{{\ol\CO}}
\newcommand{\rP}{{\ol P}}

\newcommand{\rDelta}{{\ol\Delta}}
\newcommand{\rsim}{\operatorname{\ol\sim}}

\newcommand{\imag}{\mathrm{im}}
\newcommand{\real}{\mathrm{re}}

\newcommand{\comment}[1]{}

\begin{document}

\pagenumbering{arabic}
\title[]{On the subgeneric restricted blocks of affine category $\CO$ at the critical level}
\author[]{Peter Fiebig}
\begin{abstract} We determine the endomorphism algebra of a projective generator in a subgeneric restricted block of the critical level category $\CO$ over an affine Kac--Moody algebra. 
\end{abstract}

\maketitle
\section{Introduction}

This article complements the results of the articles \cite{AF08} and \cite{AF09}. There we studied the structure of restricted critical level representations for affine Kac--Moody algebras. The two main results we obtained are the following. The first is a multiplicity formula for restricted Verma modules with a {\em subgeneric} critical highest weight, and the second is a linkage principle together with a block decomposition for the restricted category $\CO$. In this article we use these results in order to describe the categorical structure of the subgeneric restricted blocks of $\CO$.

We would like to be able to describe the structure of  all restricted blocks and to establish more general multiplicity and character formulas. Generically, a restricted critical level block contains a unique simple object which is, moreover, projective. This implies that such a block is equivalent to the category of $\DC$-vector spaces. The next simplest situation is already much more involved. The {\em subgeneric} blocks contain infinitely many simples. Every subgeneric restricted Verma module has a two-step Jordan--H\"older filtration, and the restricted version of BGGH-reciprocity (see \cite{AF09}) tells us that a restricted subgeneric indecomposable projective object is a non-split extension of two Verma modules. In this note we describe the endomorphism algebra of a projective  generator in such a subgeneric block.

\section{Affine Kac--Moody algebras}
 In this section we collect the main structural results on affine Kac-Moody algebras. 
Let $\fg$ be a simple complex Lie algebra and let $\hfg$ be the corresponding affine Kac--Moody algebra. As a vector space, $\hfg=\fg\otimes_\DC\DC[t,t^{-1}]\oplus \DC K\oplus\DC D$, and the Lie bracket on $\hfg$ is determined by the rules
\begin{align*}
[x\otimes t^m, y\otimes t^n] &= [x,y]\otimes t^{m+n}+m\delta_{m,-n} (x,y)K,\\
[K,\hfg]&=\{0\},\\
[D,x\otimes t^n]&=nx\otimes t^n,
\end{align*}
where $x$ and $y$ are elements of  $\fg$, $m$ and $n$ are integers, $\delta_{a,b}$ is the Kronecker symbol, and $(\cdot,\cdot)\colon \fg\times\fg\to\DC$ denotes the Killing form on $\fg$.

Let $\fh\subset\fg$ be a Cartan subalgebra and $\fb\subset \fg$   a Borel subalgebra containing $\fh$. Then
\begin{align*}
\hfh&:=\fh\oplus\DC K\oplus\DC D,\\
\hfb&:=\fg\otimes t\DC[t]\oplus\fb\oplus\DC K\oplus\DC D
\end{align*}
denote the corresponding affine Cartan  and Borel subalgebras of $\hfg$, respectively.

\subsection{Affine roots}
We denote by $V^\star$ the dual of a vector space $V$. Let $R \subset \fhd$ be the set of roots of $\fg$ with respect to $\fh$. We consider $\fhd$ as a subspace in $\hfhd$ by letting each $\lambda\in\fhd$ act trivially on $\DC K\oplus\DC D$. 
We define  $\delta\in \hfhd$ by 
\begin{align*}
\delta(\fh\oplus \DC K) & = \{0\}, \\ 
\delta(D) & = 1.
\end{align*}
The set $\hR\subset \hfhd$ of roots of $\hfg$ with respect to $\hfh$ is 
$$
\hR=\{\alpha+n\delta \mid \alpha\in R,n\in\DZ\}\cup\{n\delta\mid n\in \DZ, n\ne 0\}.
$$
The subsets 
\begin{align*}
\hR^{\real} &:=\{\alpha+n\delta \mid \alpha\in R, n\in\DZ\},\\
\hR^{\imag} &:=  \{n\delta\mid n\in \DZ, n\ne 0\}
\end{align*}
are called the sets of {\em real} roots and of {\em imaginary} roots,
resp.

We denote by $ R^+\subset  R$ the positive (finite) roots, i.e.~the set of roots of $\fb$ with respect
to $\fh$. Then the set of positive affine roots, i.e.~the set of roots of $\hfb$ with respect to $\hfh$, is 
$$
\hR^{+}:=\{\alpha+n\delta\mid \alpha\in R, n\ge 1\}\cup
 R^{+}\cup \{n\delta\mid n\ge 1\}.
$$

The partial order ``$\le$'' on $\hfhd$ is defined as follows. We have $\lambda\le\mu$ if $\mu-\lambda$ is a sum of positive affine roots. 

\subsection{The invariant bilinear form}
There is an extension of the  Killing form $(\cdot,\cdot)$ on $\fg$ to a symmetric bilinear form on $\hfg$. It is determined by the following: 
\begin{align*}
(x\otimes t^n, y\otimes t^m) &= \delta_{n,-m} (x,y), \\
(K, \fg\otimes_\DC \DC[t,t^{-1}]\oplus \DC K) &= \{0\}, \\
(D, \fg\otimes_\DC \DC[t,t^{-1}]\oplus \DC D) &= \{0\}, \\
(K,D) &= 1
\end{align*}
for $x,y\in \fg$ and $m,n\in\DZ$. This form  is again non-degenerate and {\em invariant}, i.e.~
it satisfies   $([x,y],z)=(x,[y,z])$ for all
$x,y,z\in\hfg$. Moreover, it induces a non-degenerate bilinear form on the Cartan
subalgebra $\hfh$ and hence yields an isomorphism $\hfh\stackrel{\sim}\to\hfhd$.
 We get an induced symmetric
non-degenerate bilinear form on the  dual $\hfhd$, which we again denote by the symbol $(\cdot,\cdot)$. 

\begin{remark}\label{rem-deltaK} The definitions immediately imply that the isomorphism $\hfh\to\hfhd$ from above  maps $K$ to $\delta$, i.e.~for any $\lambda\in\hfhd$ we have
$$
\lambda(K)=(\delta,\lambda).
$$
In particular, $(\delta,\gamma)=0$ for any $\gamma\in\hR$.
\end{remark}

\subsection{The Weyl group}  

For each real affine root  $\alpha+n\delta$ we have $(\alpha+n\delta,\alpha+n\delta)=(\alpha,\alpha)\ne 0$, hence we can define the reflection 
\begin{align*}
s_{\alpha,n}\colon \hfhd&\to\hfhd\\
\lambda&\mapsto \lambda-2\frac{(\lambda,\alpha+n\delta)}{(\alpha,\alpha)}(\alpha+n\delta).
\end{align*}
This is  a reflection as it stablizes the hyperplane $(\cdot,\alpha+n\delta)=0$ and maps $\alpha+n\delta$ to its inverse.

We denote by $\hCW\subset \GL(\hfhd)$ the affine Weyl group, i.e.~the
subgroup generated by the reflections $s_{\alpha,n}$ for
$\alpha\in R$ and $n\in\DZ$. The subgroup $\CW\subset\hCW$ generated by the
reflections $s_{\alpha,0}$ with $\alpha\in R$  leaves the subset
$\fhd\subset \hfhd$ stable and can be identified with the Weyl
group of $\fg$. 

 Let $\rho\in\hfhd$ be an element that takes the value 1 on any simple affine coroot. This element is defined only up to addition of a multiple of $\delta$. Nevertheless, nothing in what follows will depend on the value of $\rho$ on $D$ in an essential way.

 The {\em dot-action} $\hCW\times\hfhd\to\hfhd$,
  $(w,\lambda)\mapsto w.\lambda$, of the affine Weyl group on
  $\hfhd$ is obtained by shifting the linear action in such a way that
  $-\rho$ becomes a fixed point, i.e.~ it is given by
$$
w.\lambda:=w(\lambda+\rho)-\rho
$$
for $w\in\hCW$ and $\lambda\in \hfhd$. Note that since
$(\delta,\alpha+n\delta)=0$ we have 
$s_{\alpha,n}(\delta)=\delta$ for all
$\alpha+n\delta\in \hR^{\real}$. Hence 
$w(\delta)=\delta$ for all $w\in\hCW$ (so the dot-action is
independent of the choice of $\rho$).

\section{The affine category $\CO$}\label{sec-affO}
We denote by $\hCO$ the full subcategory of the category of representations of $\hfg$ that contains an object $M$ if and only if it has the following properties:
\begin{itemize}
\item $M$ is semisimple under the action of $\hfh$,
\item $M$ is locally finite under the action of $\hfb$.
\end{itemize}
The first condition means that $M=\bigoplus_{\lambda\in\hfhd} M_\lambda$, where $M_\lambda=\{m\in M\mid h.m=\lambda(h)m\text{ for all $h\in\hfh$}\}$, and the second that each $m\in M$ is contained in a finite-dimensional sub-$\hfb$-module of $M$. 

For any $\lambda\in\hfhd$ we denote by $\Delta(\lambda)$ the Verma module with highest weight $\lambda$, and by $L(\lambda)$ its unique irreducible quotient. The $L(\lambda)$ for $\lambda\in\hfhd$ are a system of representatives of the simple objects in the category $\hCO$. 

For an object $M$ of $\hCO$ and a simple object $L$ in $\hCO$ we denote by $[M:L]\in\DN$ the corresponding Jordan--H\"older multiplicity, whenever this makes sense (see \cite{DGK}). In general, we write $[M:L]\ne 0$ if $L$ is isomorphic to a subquotient (i.e.~a quotient of a subobject) of $M$. 

\subsection{Projective objects in $\hCO$}
In order to describe the categorical structure of $\hCO$ we want to describe the endomorphism algebra of a projective generator. Now $\hCO$ does not contain enough projectives. Fortunately, it is filtered by ``truncated subcategories'' that do contain enough projectives, which for us is good enough. 

In order to define the truncated subcategories, we need the following topology on $\hfhd$.
\begin{definition} A subset $\CJ$ of $\hfhd$ is called {\em open} if it is  downwardly closed with respect to the partial order ``$\le$'', i.e.~if it satisfies the following condition: If $\lambda\in \CJ$ and $\mu<\lambda$, then $\mu\in \CJ$.
An open subset $\CJ$ of $\hfhd$ is called {\em bounded} (rather, {\em locally bounded from above}), if for any $\lambda\in\CJ$, the set $\{\nu\in\CJ\mid \nu>\lambda\}$ is finite.
\end{definition} 
 Now we can define the truncated subcategories.
\begin{definition} Let $\CJ\subset \hfhd$ be open. Then $\hCO^\CJ$ is the full subcategory of $\hCO$ that contains all objects $M$
with the property that $M_\lambda\ne\{0\}$ implies $\lambda\in\CJ$.
\end{definition}
For any $\lambda\in\hfhd$ the set $\{\mu\in\hfhd\mid\mu\le\lambda\}$ is open. We use the notation $\hCO^{\le\lambda}$ instead of $\hCO^{\{\mu\in\hfhd\mid\mu\le\lambda\}}$. Note that $L(\lambda)$ is contained in $\hCO^{\CJ}$ if and only if $\lambda\in\CJ$. The inclusion functor $\hCO^\CJ\to\hCO$ has a left adjoint that we denote by $M\mapsto M^\CJ$. It is defined as follows: Let $\CI=\hfhd\setminus\CJ$ be the closed complement of $\CJ$ and let $M_\CI\subset M$ be the submodule generated by all weight spaces $M_\lambda$ with $\lambda\in\CI$. Then set
$$
M^\CJ:=M/M_\CI.
$$
This definition clearly is functorial.
We will need the following notion.
\begin{definition}
Let $M\in\hCO$. We say that $M$ {\em admits a Verma flag} if there is a finite filtration
$$
0=M_0\subset M_1\subset\dots \subset M_n=M
$$
with $M_i/M_{i-1}\cong \Delta(\mu_i)$ for some $\mu_1,\dots,\mu_n\in\hfhd$.
\end{definition}

In case $M$ admits a Verma flag, the {\em Verma multiplicity} $(M:\Delta(\mu_i))=\#\{i\in\{1,\dots,n\}\mid \mu_i=\mu\}$ is independent of the chosen filtration. The following is proven in \cite{RCW} (see also \cite{F12}).
\begin{theorem}\label{thm-BGG} Let $\CJ\subset \fhd$ be open and bounded and let $\lambda\in\CJ$.
\begin{enumerate}
\item There exists a projective cover $P^\CJ(\lambda)\to L(\lambda)$ in $\hCO^\CJ$ and the object $P^\CJ(\lambda)$ admits a Verma flag.
\item BGGH-reciprocity 
$$
(P^\CJ(\lambda):\Delta(\mu))=\begin{cases}
[\Delta(\mu):L(\lambda)], &\text{ if $\mu\in\CJ$},\\
0, &\text{ if $\mu\not\in\CJ$}
\end{cases}
$$
holds for the Jordan-H\"older and Verma multiplicities.
\item If $\CJ^\prime\subset\CJ$ is an open subset, then $P^\CJ(\lambda)^{\CJ^\prime}\cong P^{\CJ^\prime}(\lambda)$.
\item For any $M\in\hCO^\CJ$ such that $[M:L(\lambda)]$ is finite we have
$$
\dim_\DC\Hom_{\hCO}(P^\CJ(\lambda), M)=[M:L(\lambda)].
$$
\end{enumerate}
\end{theorem}

\subsection{The block decomposition of category $\hCO$}
We quickly summarize the basic facts about the block decomposition of category $\hCO$. Recall that the simple isomorphism classes are parametrized by $\hfhd$ by means of their highest weight. The block decomposition in particular yields a partition of the simple isomorphism classes. In terms of their parameters, this partition is given as follows.

\begin{definition}\label{def-equcl} Let ``$\sim$'' be the equivalence relation on $\hfhd$ that is generated by the following. We have $\lambda\sim\mu$ if there exists a positive affine root $\gamma\in \hR^+$ and $n\in\DZ$ such that 
 $2(\lambda+\rho,\gamma)=n(\gamma,\gamma)$ and  $\mu=\lambda-n\gamma$.
\end{definition}

For an equivalence class $\Lambda\subset\hfhd$ with respect to ``$\sim$'' we let $\hCO_\Lambda$ be the full subcategory of $\hCO$ that contains all objects $M$ with the property that $[M:L(\lambda)]\ne 0$ implies $\lambda\in\Lambda$. The linkage principle (see \cite{KK}) together with BGGH-reciprocity mentioned above now yields the following.

\begin{theorem} \label{thm-blockdec}The functor
\begin{align*}
\prod_{\Lambda\in\hfhd/_{ \sim}}\hCO_\Lambda&\to\hCO,\\
\{M_\Lambda\}&\mapsto\bigoplus_{\Lambda\in\hfhd/_{ \sim}} M_\Lambda
\end{align*}
is an equivalence of categories. 
\end{theorem}

\subsection{The level}
As  the central line $\DC K$ of $\hfg$ is contained in $\hfh$, it acts on each object $M$ of $\hCO$ by  semisimple endomorphisms. 
For each $k\in\DC$, we denote by $M_k$ the eigenspace of the action of $K$ with eigenvalue $k$.  We define $\hfhd_k\subset\hfhd$ as the affine hyperplane that contains all $\lambda$ with $\lambda(K)=k$, so
$M_k=\bigoplus_{\lambda\in\hfhd_k} M_\lambda$. 
The eigenspace decomposition $M=\bigoplus_{k\in\DC} M_k$ is a decomposition into sub-$\hfg$-modules of $M$. 
When $M=M_k$ for some $k$ we call $k$  the {\em level} of $M$, and we let $\hCO_k$ be the full subcategory of $\hCO$ that contains all objects of level $k$.

If $\lambda\sim\mu$, then $\lambda$ and $\mu$ differ by a sum of affine roots. As $\gamma(K)=0$ for any $\gamma\in\hR$, for each equivalence class $\Lambda$ there is a $k=k_\Lambda$ with $\Lambda\subset\hfhd_k$, i.e.~each block $\hCO_\Lambda$ determines a level.

There is a specific level that we denote by  ``$\crit$'' and  that is distinguished in more than one respect. It is $\crit=-\rho(K)$ (this is another instance of the above mentioned independence of the choice of $\rho$). In the usual normalization, this is $-h^\vee$, where $h^\vee$ denotes the dual Coxeter number of $\fg$. The elements in $\hfhd_{\crit}$ are called {\em critical weights}, and, anlogously, we call an equivalence class $\Lambda$ {\em critical} if $\Lambda\subset\hfhd_{\crit}$. 

\subsection{The structure of equivalence classes}
For any two affine roots $\alpha$ and $\beta$ we have $2(\beta,\alpha)\in\DZ(\alpha,\alpha)$. Since equivalent weights differ by a sum of affine roots, this implies, 
$$
\{\alpha\in\hR\mid 2(\lambda+\rho,\alpha)\in\DZ(\alpha,\alpha)\}=\{\alpha\in\hR\mid 2(\mu+\rho,\alpha)\in\DZ(\alpha,\alpha)\}
$$
whenever $\lambda\sim\mu$. 
If $\Lambda$ is a $\sim$-equivelance class, we can hence define
$$
\hR(\Lambda):=\{\alpha\in\hR\mid 2(\lambda+\rho,\alpha)\in\DZ(\alpha,\alpha)\text{ for some (all) $\lambda\in\Lambda$}\}.
$$

\begin{lemma} Let $\Lambda$ be a $\sim$-equivalence class. Then the following are equivalent:
\begin{enumerate}
\item $\delta\in\hR(\Lambda)$,
\item $\DZ\delta\subset\hR(\Lambda)$,
\item $\Lambda$ is critical, i.e.~$\Lambda\subset \hfhd_{\crit}$.
\end{enumerate}
\end{lemma}
\begin{proof} Note that $(\delta,\delta)=0$. Let $\lambda\in\Lambda$. We have $\delta\in\hR(\Lambda)$ if and only if $(\lambda+\rho,\delta)=0$. This is the case if and only if $(\lambda+\rho,n\delta)=0$ for all $n\in\DZ$, i.e.~if and only if $\DZ\delta\subset\hR(\Lambda)$. Finally, $(\lambda+\rho,\delta)=(\lambda+\rho)(K)$ by Remark \ref{rem-deltaK}, and this equals $0$ if and only if $\lambda(K)=-\rho(K)$, i.e.~if and only if $\lambda$ is critical.
\end{proof}

\begin{lemma}\label{lemma-critintroots} Suppose that $\lambda$ is critical and $\alpha\in R$. Then the following are equivalent.
\begin{enumerate}
\item $\alpha+n\delta\in\hR(\lambda)$ for some $n\in\DZ$,
\item $\alpha+n\delta\in\hR(\lambda)$ for all $n\in\DZ$.
\end{enumerate}
\end{lemma}
\begin{proof} Note that $(\alpha+n\delta,\alpha+n\delta)=(\alpha,\alpha)$, as $(\delta,\gamma)=0$ for any affine root $\gamma$ by Remark \ref{rem-deltaK}. As $\lambda$ is critical, $(\lambda+\rho,\delta)=0$. Hence, both statements are equivalent to $2(\lambda+\rho,\alpha)\in\DZ(\alpha,\alpha)$.
\end{proof}

For any $\sim$-equivalence class $\Lambda$ we define
$$
\hCW(\Lambda)=\langle s_{\alpha,n}\mid \alpha+n\delta\in\hR(\Lambda)\rangle.
$$

\begin{lemma}\label{lemma-struceq} \begin{enumerate}
\item Suppose that $\Lambda$ is not-critical. Then 
$$
\Lambda=\hCW(\Lambda).\lambda
$$
for any $\lambda\in\Lambda$.
\item Suppose that $\Lambda$ is critical. Then 
$$
\Lambda=\hCW(\Lambda).\lambda+\DZ\delta
$$
for any $\lambda\in\Lambda$.
\end{enumerate}
\end{lemma}
\begin{proof} Let  $\lambda,\mu\in\hfhd$, $\gamma\in\hR^+$ and $n\in\DZ$ be as in Definition \ref{def-equcl}. Then $\lambda-\mu=n\gamma$ and $2(\lambda+\rho,\gamma)=n(\gamma,\gamma)$. Note that if $\gamma$ is real, then $\lambda=s_\gamma.\mu$.  If $\gamma$ is imaginary, then $\gamma=m\delta$ for some $m\ne 0$ and $(\gamma,\gamma)=0$ and $(\lambda+\rho,\delta)=0$, which implies $\lambda+n\delta\sim\lambda$ for all $n\in\DZ$. This, together with the fact that $\hR(\lambda)=\hR(\mu)$, implies the statements.
\end{proof}

\section{Extensions of neighbouring Verma modules}
In this section we collect some results about extensions of $\Delta(\lambda)$ and $\Delta(\mu)$ in $\hCO$, where $\lambda$ and $\mu$ are ``neighbouring''. By this we mean the following.
\begin{definition} The elements $\lambda,\mu\in\hfhd$ are called {\em neighbouring} if the following conditions are satisfied:
\begin{enumerate}
\item There is $\alpha\in\hR^+\cap\hR^{\re}$ and $n\in\DN$ with $2(\lambda,\alpha)=n(\alpha,\alpha)$ and $\mu=\lambda+n\alpha$. In particular, $\lambda\sim\mu$ and $\lambda<\mu$.
\item There is no $\nu\in\hfhd$ that is $\sim$-equivalent to both $\lambda$ and $\mu$ with $\lambda<\nu<\mu$.
\end{enumerate}
\end{definition}
Our first result is the following: 
\begin{lemma}\label{lem-neimult} Suppose that $\lambda$ and $\mu$  are neighbouring and $\lambda < \mu$. Then $[\Delta(\mu):L(\lambda)]=1$.
\end{lemma}
\begin{proof} Let
$$
\Delta(\mu)=M_0\supset M_1\supset M_2\supset\dots
$$
be the Jantzen filtration (for this and the sum formula below, see \cite{KK}). Then $\Delta(\mu)/M_1\cong L(\mu)$. The Jantzen sum formula says
$$
\sum_{i>0}\cha\, M_i=\sum_{\substack{\alpha\in \hR^+, n\in\DN, \\2(\mu+\rho,\alpha)=n(\alpha,\alpha)}} \cha \Delta(\mu-n\alpha),
$$
where the roots should be counted with their multiplicities (i.e. the imaginary roots should be counted $\rk\,\fg$-times). 
Now on the right hand side, $\Delta(\lambda)$ occurs exactly once, and otherwise only $\Delta(\nu)$ appear with $\nu\not\ge\lambda$. Hence $\left[M_1:L(\lambda)\right]=1$, so $[\Delta(\mu):L(\lambda)]=1$.
\end{proof}

\begin{lemma}\label{lemma-neiext} Suppose that $\lambda$ and $\mu$  are neighbouring and $\lambda<\mu$. Then $\dim_\DC\Ext^1_{\hCO}(\Delta(\lambda),\Delta(\mu))=1$. 
\end{lemma}
\begin{proof} It is enough to calculate $\Ext^1$ in the subcategory $\hCO^{\le\mu}$ of $\hCO$. By Lemma \ref{lem-neimult} and  BGGH-reciprocity we have $(P(\lambda)^{\le\mu}:\Delta(\mu))=(P(\lambda)^{\le\mu}:\Delta(\lambda))=1$ and all other multiplicities are $0$. Hence there is a short exact sequence
$$
0\to\Delta(\mu)\cong P(\mu)^{\varle\mu}\to P(\lambda)^{\le\mu}\to\Delta(\lambda)\to 0
$$
which is already  a projective resolution of $\Delta(\lambda)$ in $\hCO^{\le\mu}$. Applying $\Hom_{\hCO^{\varle\mu}}(\cdot,\Delta(\mu))$ to $0\to\Delta(\mu)\to P(\lambda)^{\le\mu}\to 0$ yields 
$$
0\to\Hom_{\hCO^{\varle\mu}}(P(\lambda)^{\le\mu},\Delta(\mu))\to\Hom_{\hCO^{\varle\mu}}(\Delta(\mu), \Delta(\mu))\to 0.
$$
Both $\Hom$-spaces are one-dimensional (the first again by Lemma \ref{lem-neimult}), and each non-zero homomorphism $P(\lambda)^{\varle\mu}\to\Delta(\mu)$ factors through an inclusion $\Delta(\lambda)\to\Delta(\mu)$, hence has $\Delta(\mu)\subset P(\lambda)^{\varle\mu}$ in its kernel. So the middle homomorphism in the above sequence vanishes, so the dimension of  $\Ext^1_{\hCO^{\le\mu}}(\Delta(\lambda),\Delta(\mu))$ is $1$.  \end{proof}

We denote by $Z(\lambda,\mu)\in\hCO$ the (unique up to isomorphism) non-split extension of $\Delta(\mu)$ and $\Delta(\lambda)$ for neighbouring $\lambda$ and $\mu$.

\begin{lemma} Suppose that $\lambda$ and $\mu$ are neighbouring and that $\lambda<\mu$. Then $P^{\le\mu}(\lambda)\cong Z(\lambda,\mu)$.
\end{lemma}
\begin{proof} By  BGGH-reciprocity, $P^{\le\mu}(\lambda)$ has a two-step Verma flag with subquotients isomorphic to $\Delta(\mu)$ and $\Delta(\lambda)$. This filtration is non-split, as $\Delta(\mu)$ is not a quotient of $P^{\le\mu}(\lambda)$, since $L(\mu)$ is not. Hence the claim.
\end{proof}

Note that for any $\lambda,\mu\in\hfhd$ we have $\dim_\DC\Hom_{\hCO}(\Delta(\lambda),\Delta(\lambda+n\delta))\le[\Delta(\lambda+n\delta):L(\lambda)]$. We now study the particular situation in which this is an equality. 
\begin{lemma}\label{lemma-main1} Suppose that $\lambda$ and $\mu$ are neighbouring and $\lambda<\mu$. Let $n>0$ and suppose that 
$$
\dim_\DC\Hom_{\hCO}(\Delta(\lambda),\Delta(\lambda+n\delta))=[\Delta(\lambda+n\delta):L(\lambda)].
$$
Then every homomorphism $Z(\lambda,\mu)\to \Delta(\lambda+n\delta)$ factors through a homomorphism $\Delta(\lambda)\to\Delta(\lambda+n\delta)$.
\end{lemma}
\begin{proof} Let $\CJ$ be open and bounded and suppose it contains all relevant weights $\lambda$, $\mu$ and $\lambda+n\delta$. By the previous lemma, we have a surjection $P^\CJ(\lambda)\to Z(\lambda,\mu)$. So the chain of surjections
$$
P^\CJ(\lambda)\sur Z(\lambda,\mu)\sur \Delta(\lambda)
$$
induces a chain of injections
\begin{align*}
\Hom_\hCO(\Delta(\lambda),\Delta(\lambda+n\delta))&\inj \Hom_\hCO(Z(\lambda,\mu),\Delta(\lambda+n\delta))\\
&\inj\Hom_\hCO(P^\CJ(\lambda),\Delta(\lambda+n\delta)).
\end{align*}
Now the dimension of the space on the right is $[\Delta(\lambda+n\delta):L(\lambda)]$, as $P^\CJ(\lambda)$ is a projective cover of $L(\lambda)$ in $\hCO^\CJ$, so our assumptions imply that the above injections are bijections. This proves the lemma.\end{proof}
\subsection{The tilting equivalence}

Let $\CM$ be the full subcategory of $\hCO$ that contains all objects that admit a Verma flag. 

\begin{theorem} [{\cite[Korollar 2.3]{Soe98}}] \label{thm-tilting} There is an equivalence $t\colon \CM\to\CM^{opp}$ that maps short exact sequences to short exact sequences and satisfies
$$
t(\Delta(\lambda))\cong \Delta(-2\rho-\lambda).
$$
\end{theorem}
(Note that this statement does depend on the choice of $\rho$.)

Note that $t$ stabilizes $\hCO_{\crit}$, as $(-2\rho-\lambda)(K)=2\,\crit-\crit=\crit$ for all $\lambda\in\hfhd_{\crit}$.

\begin{lemma} Suppose that $\lambda$ and $\mu$ are neighbouring and that $\lambda<\mu$.  Then $t Z(\lambda,\mu)\cong Z(-2\rho-\mu, -2\rho-\lambda)$.
\end{lemma}
\begin{proof} Applying the tilting equivalence to the short exact sequence
$$
0\to \Delta(\mu)\to Z(\lambda,\mu)\to\Delta(\lambda)\to 0
$$
yields a non-split short exact sequence
$$
0\to \Delta(-2\rho-\lambda)\to t Z(\lambda,\mu)\to \Delta(-2\rho-\mu)\to 0. 
$$
Lemma \ref{lemma-neiext} now immediately implies the statement. 
\end{proof}

Applying the tilting equivalence to the statement of Lemma \ref{lemma-main1} and using the previous lemma we obtain:

\begin{lemma}\label{lemma-main2} Suppose that $\lambda$, $\mu$ are neighbouring and $\lambda<\mu$. Suppose  that
$$
\dim_\DC\Hom_{\hCO}(\Delta(-2\rho-\mu),\Delta(-2\rho-\mu+n\delta))=[\Delta(-2\rho-\mu+n\delta):L(-2\rho-\mu)].
$$
Then every homomorphism $\Delta(\mu-n\delta)\to Z(\lambda,\mu)$ factors through a homomorphism $\Delta(\mu-n\delta)\to \Delta(\mu)$.
\end{lemma}

\section{Restricted critical level representations}
 
We will now define the subcategory $\rCO_{\crit}$ of $\hCO_{\crit}$ that contains all {\em restricted} representations, and we will review structural results on this subcategory that resemble the ones we discussed in Section \ref{sec-affO} (references for the following are \cite{AF08} and \cite{AF09}).

\subsection{The Feigin--Frenkel center}

Let us denote by 
$$
\CZ=\bigoplus_{n\in\DZ}\CZ_n=\DC[p_s^{(i)}, i=1,\dots,\rk\,\fg, s\in\DZ]
$$ 
the polynomial ring (of infinite rank)  constructed from the center of the critical level vertex algebra (see \cite[Section 5]{AF08}). We consider it as a $\DZ$-graded algebra with $p_s^{(i)}$ being homogeneous of degree $s$. 

The algebra $\CZ$  acts on objects in $\hCO_{\crit}$ in the following way. 
The simple highest weight module $L(\delta)$ is invertible, i.e.~it is one-dimensional and $L(\delta)\otimes_\DC L(-\delta)$ is isomorphic to the trivial $\hfg$-module $L(0)$. Hence, the functor
\begin{align*}
T\colon \hCO&\to\hCO\\
M&\mapsto M\otimes_\DC L(\delta)
\end{align*}
is an equivalence with inverse $M\mapsto M\otimes_\DC L(-\delta)$. As the level of a tensor product equals the sum of the levels of its factors, and as $L(\delta)$ is of level zero, the functor $T$ preserves the subcategories $\hCO_k$ for any $k\in\DC$. We will henceforth restrict it to $\hCO_\crit$.

Note that  $\fg\otimes_\DC\DC[t,t^{-1}]\otimes\DC K$ acts trivially on $L(\delta)$, while, as $\delta(D)=1$, the grading element $D$ acts as the identity. 

\begin{lemma}[{\cite{AF08}}] Let $z\in\CZ_n$. For any $M\in\hCO_\crit$,  $z$ defines a homomorphism $z^M\colon T^n M\to M$.
\end{lemma}

\subsection{Restricted representations}

We define restricted representations by the following vanishing condition on the action of $\CZ$:

\begin{definition} An object $M\in\hCO_\crit$ is called {\em restricted} if for any $n\ne 0$ and any $z\in \CZ_n$ we have that $z^M$ is zero.
\end{definition}
We denote by $\rCO_{\crit}$ the full subcategory of $\hCO_{\crit}$ that contains all restricted objects. There is a functor $(\cdot)^\res\colon \hCO_{\crit}\to\rCO_{\crit}$ that is left adjoint to the inclusion $\rCO_{\crit}\subset\hCO_{\crit}$. It is defined as 
$$
M^{\res}:=M/ M^\prime,
$$
where $M^\prime$ is the submodule of $M$ that is generated by the images of all homomorphisms $z^M$ with $z\in\CZ_n$ and $n\ne 0$. 

For any $\lambda\in\hfhd$, the restricted Verma module with highest weight $\lambda\in\hfhd$ is defined as
$$
\rDelta(\lambda):=\Delta(\lambda)^{\res}.
$$
The next definition is the obvious restricted version of the earlier notion of a Verma flag.
\begin{definition}
Let $M\in\hCO_{\crit}$. We say that $M$ {\em admits a restricted Verma flag} if there is a finite filtration
$$
0=M_0\subset M_1\subset\dots \subset M_n=M
$$
with $M_i/M_{i-1}\cong \rDelta(\mu_i)$ for some $\mu_1,\dots,\mu_n\in\hfhd_{\crit}$.
\end{definition}

In case $M$ admits a restricted Verma flag, the {\em restricted Verma multiplicity} $\left(M:\rDelta(\mu_i)\right)=\#\{i\in\{1,\dots,n\}\mid \mu_i=\mu\}$ is again independent of the chosen filtration. 

\subsection{Restricted projective objects}
For any open bounded subset $\CJ$ of $\hfhd_{\crit}$ set $\rCO_{\crit}^\CJ:=\rCO_{\crit}\cap\hCO_{\crit}^\CJ$.
The following is an analogue of Theorem \ref{thm-BGG} in the restricted setting.
\begin{theorem}[\cite{AF09}, see also {\cite[Theorem 4.3 and Theorem 5.4]{F12}}] Let $\CJ\subset\hfhd_{\crit}$ be an open bounded subset and let $\lambda\in\CJ$.

\begin{enumerate}
\item There exists a projective cover $\rP^\CJ(\lambda)\to L(\lambda)$ of $L(\lambda)$ in $\rCO_\crit^\CJ$ and the object $\rP^\CJ(\lambda)$ admits a restricted Verma flag.
\item  For the multiplicities we have
$$
\left(\rP^\CJ(\lambda):\rDelta(\mu)\right)=\begin{cases}
\left[\rDelta(\mu):L(\lambda)\right], &\text{ if $\mu\in\CJ$}\\
0, &\text{ if $\mu\not\in\CJ$}.
\end{cases}
$$
\item For any open subset $\CJ^\prime\subset\CJ$ we have $\ol P^\CJ(\lambda)^{\CJ^\prime}\cong \ol P^{\CJ^\prime}(\lambda)$.
\item For any $M\in\rCO^\CJ_{\crit}$  such that $[M:L(\lambda)]$ is finite we have
$$
\dim_\DC \Hom_{\rCO_\crit}(\ol P^\CJ(\lambda),M)=[M:L(\lambda)].
$$
\end{enumerate}
\end{theorem}

In \cite{AF09} we showed that one obtains $\rP^\CJ(\lambda)$ from $P^\CJ(\lambda)$ by applying the restriction functor, i.e.
$$
\rP^\CJ(\lambda)=P^\CJ(\lambda)^{\res}
$$ 
for any $\lambda\in\CJ$.

\subsection{The restricted block decomposition}

The block decomposition $\hCO=\prod_{\Lambda\in\hfhd/_{ \sim}}\hCO_\Lambda$ of Theorem \ref{thm-blockdec} clearly induces a block decomposition of $\rCO_{\crit}$.  It turns out that a component $\rCO_{\crit}\cap\hCO_\Lambda$  is, in general, no longer indecomposable. Note that the following definition is a version of Definition \ref{def-equcl}, that only utilizes real affine roots instead of all affine roots.

\begin{definition} Let ``$\rsim$'' be the equivalence relation on $\hfhd$ that is generated by the following. We have $\lambda\rsim\mu$ if there exists a positive {\em real} root $\gamma\in \hR^{\real}\cap\hR^+$ and $n\in\DZ$ such that 
 $2(\lambda+\rho,\gamma)=n(\gamma,\gamma)$ and 
 $\mu=\lambda-n\gamma$.
\end{definition}
Clearly, ``$\rsim$'' is a finer equivalence relation than ``$\sim$'' and it coincides with ``$\sim$'' on the affine hyperplanes $\hfhd_k$ for all $k\ne\crit$.

For a $\rsim$-equivalence class $\Gamma\subset\hfhd_{\crit}$ we let $\rCO_\Gamma$ be the full subcategory of $\rCO_{\crit}$ that contains all objects $M$ with the property that $[M:L(\lambda)]\ne 0$ implies $\lambda\in\Gamma$. Then we have the following analogue of Theorem \ref{thm-blockdec}.

\begin{theorem}[{\cite{AF09}}] The functor
\begin{align*}
\prod_{\Gamma\in\hfhd_{\crit}/_{ \rsim}}\rCO_\Gamma&\to\rCO_{\crit},\\
\{M_\Gamma\}&\mapsto\bigoplus_{\Gamma\in\hfhd_{\crit}/_{{\rsim}}} M_\Gamma
\end{align*}
is an equivalence of categories. 
\end{theorem}

For the restricted equivalence relation, we get the  following analogue of Lemma \ref{lemma-struceq}, (1), which is proven using the analogous arguments.
\begin{lemma} Let $\Gamma$ be a critical $\rsim$-equivalence class. Then 
$$
\Gamma=\hCW(\Gamma).\lambda
$$
for any $\lambda\in\Gamma$.
\end{lemma}

\section{The structure of subgeneric critical restricted blocks}
In this section we describe the structure of $\rCO_\Gamma$ in the case that $\Gamma$ is a {\em subgeneric} $\rsim$-equivalence class. 

\subsection{Subgeneric critical equivalence classes}
Let $\alpha\in R^+$ be a {\em finite} positive root, and let $\hCW^\alpha\subset\hCW$ be the subgroup  generated by the reflections $s_{\alpha,n}$ with $n\in\DZ$. Then $\hCW^\alpha$ is the affine Weyl group of type $A_1$, and it is generated by $s_{\alpha,0}$ and $s_{\alpha,-1}$. 

\begin{definition} Let $\gamma\in\hfhd_{\crit}$ and let $\Gamma\subset\hfhd_{\crit}$ be its $\rsim$-equivalence class. We say that $\gamma$ is {\em $\alpha$-subgeneric} if the following holds:
\begin{enumerate}
\item $\alpha\in\hR(\gamma)$ (hence, as $\gamma$ is critical, $\alpha+n\delta\in\hR(\gamma)$ for all $n\in\DZ$ by Lemma \ref{lemma-critintroots}),
\item $\gamma$ is {\em $\alpha$-regular}, i.e. $s_{\alpha,0}.\gamma\ne \gamma$,
\item $\Gamma=\hCW^\alpha.\gamma$.
\end{enumerate}
\end{definition}

Let $\nu\in\hfhd_{\crit}$. Then $(\nu+\rho,\delta)=0$, hence
$$
s_{\alpha,n}.\nu=\nu-\frac{2(\nu+\rho,\alpha)}{(\alpha,\alpha)}(\alpha+n\delta).
$$
We call $\nu$ {\em $\alpha$-dominant}, if $(\nu+\rho,\alpha)\in\DZ_{\ge 0}$. If $\nu$ is $\alpha$-dominant, then $s_{\alpha,0}.\nu\le \nu$ and $s_{\alpha,-1}\ge\nu$ (as $-\alpha+\delta$ is a positive affine root). We call $\nu$ {\em $\alpha$-antidominant}, if $(\nu+\rho,\alpha)\in\DZ_{\le 0}$. If $\nu$ is $\alpha$-antidominant, then $s_{\alpha,0}.\nu\ge\nu$ and $s_{\alpha,-1}\le\nu$. Moreover, $\nu$ is $\alpha$-dominant if and only if $s_{\alpha,0}.\nu$ is $\alpha$-antidominant, which is the case if and only if $s_{\alpha,n}.\nu$ is $\alpha$-antidominant for all $n\in\DZ$. 

So suppose that $\gamma$ is $\alpha$-subgeneric. As $\hCW^\alpha$ is generated by $s_{\alpha,0}$ and $s_{\alpha,-1}$, we conclude from the above that the equivalence class $\Gamma$ of $\gamma$ is a totally ordered set with respect to ``$\varle$''. 
For any $\nu\in\Gamma$ we define 
\begin{align*}
\alpha\uparrow\nu&:=\min\{s_{\alpha,n}.\nu\mid s_{\alpha,n}.\nu>\nu\}\\
&=\begin{cases}
s_{\alpha,-1}.\nu,&\text{ if $\nu$ is $\alpha$-dominant},\\
s_{\alpha,0}.\nu. &\text{ if $\nu$ is $\alpha$-antidominant}.
\end{cases}
\end{align*}
Then $\alpha\uparrow (\cdot)\colon\Gamma\to\Gamma$ is a bijection and we denote by $\alpha\uparrow^n(\cdot)\colon \Gamma\to\Gamma$ its $n$-fold composition for $n\in\DZ$, and we set $\alpha\downarrow^n(\cdot):=\alpha\uparrow^{-n}(\cdot)$. Then 
$$
\Gamma=\{\dots, \alpha\downarrow^2\nu, \alpha\downarrow\nu,\nu,\alpha\uparrow\nu,\alpha\uparrow^2\nu,\dots\}.
$$

\subsection{Multiplicities in the subgeneric case}

The main result of \cite{AF08} is the following multiplicity formula for $\alpha$-subgeneric $\gamma$: for any $\mu\in\hfhd_{\crit}$ we have
$$
\left[\rDelta(\gamma):L(\mu)\right]=\begin{cases}
1, &\text{ if $\mu\in\{\gamma,\alpha\downarrow\gamma\}$}, \\
0, &\text{ if $\mu\not\in\{\gamma,\alpha\downarrow\gamma\}$}.
\end{cases}
$$

Suppose that $\CJ$ is an open and bounded subset of $\hfhd_\crit$ that contains $\gamma$ and $\alpha\uparrow\gamma$. 
From  BGGH-reciprocity we then obtain that $\rP^\CJ(\gamma)$ has a two-step Verma flag with subquotients $\rDelta(\alpha\uparrow\gamma)$ and $\rDelta(\gamma)$. Clearly, $\rDelta(\alpha\uparrow\gamma)$ has to occur as a submodule, so we obtain a short exact sequence
$$
0\to \rDelta(\alpha\uparrow\gamma)\to \rP^\CJ(\gamma)\to \rDelta(\gamma)\to 0.
$$
Hence, in the subgeneric situations, the $\rP^\CJ(\gamma)$ stabilize (with respect to the partially ordered set of open subsets $\CJ$ in $\hfhd_\crit$), so there is a well-defined object
$$
\rP(\gamma):=\varprojlim_{\CJ}\rP^\CJ(\gamma)
$$
for any $\alpha$-subgeneric $\gamma$. From BGGH-reciprocity and the multiplicity statement above we  obtain that
$$
(\rP(\mu),L(\gamma))=
\begin{cases}
1, \text{ if $\mu\in\{\alpha\uparrow\gamma,\alpha\downarrow\gamma\}$},\\
2, \text{ if $\mu=\gamma$},\\
0,\text{ if $\mu\not\in\{\alpha\downarrow\gamma,\gamma,\alpha\uparrow\gamma\}$},
\end{cases}
$$
hence
$$
\dim_\DC\Hom_{\rCO}(\rP(\gamma),\rP(\mu))=
\begin{cases}
1, \text{ if $\mu\in\{\alpha\uparrow\gamma,\alpha\downarrow\gamma\}$},\\
2, \text{ if $\mu=\gamma$},\\
0, \text{ if $\mu\not\in\{\alpha\downarrow\gamma,\gamma,\alpha\uparrow\gamma\}$}.
\end{cases}
$$

\subsection{The partial restriction functor}

We will  need ``partially restricted'' objects. Let 
$$
\CZ^+:=\DC[p_s^{(i)}, i=1,\dots,\rk\,\fg, s>0].\\
$$
We set $\CZ^{+}_n:=\CZ^{+}\cap\CZ_n$. 
Then $\CZ^+$ is a positively graded subalgebra of $\CZ$.

\begin{definition} An object $M\in\hCO_\crit$ is called {\em positively restricted}  if for any $n>0$ and any $z\in \CZ^+_n$ we have that $z^M$ is zero.
\end{definition}
Note that, for example, each non-restricted Verma module $\Delta(\gamma)$ is positively restricted.
We denote by $\hCO_\crit^+$  the full subcategory of $\hCO_\crit$ that contains all positively restricted objects.
Again we have an obvious left adjoint to the inclusion functor $\hCO_{\crit}^+\subset\hCO_{\crit}$.
We let $\CZ^+ M$ be the submodule of $M$ generated by the images of all homomorphisms $z^M$ with $z\in\CZ^+_n$ and $n>0$, and we set
$$
M^+:=M/\CZ^+ M.
$$
This yields the  functor from $\hCO_\crit$ to $\hCO_\crit^+$ that is left adjoint to the inclusion functor. 

Analogously, we define
$$
\CZ^-:=\DC[p_s^{(i)}, i=1,\dots,\rk\,\fg, s<0].
$$
By replacing $\CZ^+$ by $\CZ^-$ in the definitions above, we obtain the analogous notion of {\em negatively restricted} objects, the corresponding category $\hCO_{\crit}^-$ and a functor $M\mapsto M^-$ that is left adjoint to the inclusion $\hCO_\crit^-\subset\hCO_{\crit}$. As $\CZ$ is generated by its subalgebras $\CZ^-$ and $\CZ^+$ we have
$$
M^{\res}=(M^+)^-=(M^-)^+
$$
for all $M$ in $\hCO_{\crit}$.

We now collect some results on the partial restriction functor that we need later on. 

\begin{proposition}\label{prop-parresproj} Let $\gamma\in\hfhd_{\crit}$ be $\alpha$-subgeneric and let  $\CJ\subset\hfhd_{\crit}$ be open and bounded such that $\gamma,\alpha\uparrow\gamma\in\CJ$. Then
$P^\CJ(\gamma)^+ $ is a non-split extension of $\Delta(\gamma)$ and $\Delta(\alpha\uparrow\gamma)$, hence isomorphic to $Z(\gamma,\alpha\uparrow\gamma)$.
\end{proposition}

\begin{proof}  Let $P:=P^\CJ(\gamma)$. Then $P^{\varle\alpha\uparrow\gamma}\cong P^{\varle\alpha\uparrow\gamma}(\gamma)$. Then $\gamma$ and $\alpha\uparrow\gamma$ are neighbouring, hence 
$$
\left(P^{\varle\alpha\uparrow\gamma}:\Delta(\gamma)\right)=\left(P^{\varle\alpha\uparrow\gamma}:\Delta(\alpha\uparrow\gamma)\right)=1
$$ 
and all other multiplicities are zero, so we have a short exact sequence
$$
0\to \Delta(\alpha\uparrow\gamma)\to P^{\varle\alpha\uparrow\gamma}\to \Delta(\gamma)\to 0.
$$
This is a non-split short exact sequence, as $\Delta(\alpha\uparrow\gamma)$ is not a quotient of $P^{\varle\alpha\uparrow\gamma}$. So $P^{\varle\alpha\uparrow\gamma}\cong Z(\gamma,\alpha\uparrow\gamma)$. 

Note that the kernel of the homomorphism  $P\to P^{\varle\alpha\uparrow\gamma}$ is generated by all weight spaces $P_\mu$ with $\mu\not\le\alpha\uparrow\lambda$. Now $P$ is generated by its $\gamma$-weight space, so  $\CZ^+ P$ is generated by its weight spaces $(\CZ^+ P)_{\gamma+n\delta}$ for $n>0$. As $\gamma+n\delta\not\le\alpha\uparrow\gamma$ for all $n>0$, we  obtain an induced map $P^+\to P^{\varle\alpha\uparrow\gamma}$. We claim that this map is an isomorphism, which, by the above, implies the statement of the proposition.

Clearly this map is surjective. If it is not injective, then there exists a $\mu$ with $\mu\not\le\alpha\uparrow\gamma$ and $P^+_\mu\ne 0$. Let us in this case choose a maximal such $\mu$. Then we have $(P^{\res})_\mu=((P^+)^-)_\mu\ne 0$, which contradicts the fact that $P^{\res}$ is an extension of $\Delta(\gamma)$ and $\Delta(\alpha\uparrow\gamma)$, so all its weights are $\le\alpha\uparrow\gamma$. 
\end{proof}

For simplicity we will denote $P^\CJ(\gamma)^+$ by $P(\gamma)^+$ in the following.

\subsection{Homomorphisms between projectives}

We will now construct a basis of the homomorphism space $\Hom_{\rCO}(\rP(\lambda),\rP(\mu))$ for $\lambda$ and $\mu$ in a $\alpha$-subgeneric equivalence class. We have already seen that this space is one-dimensional if $\mu\in\{\alpha\downarrow\lambda,\alpha\uparrow\lambda\}$, two-dimensional in case $\lambda=\mu$, and it is the trivial space otherwise. 

To start with, let us fix, for any $\alpha$-subgeneric $\nu$, an inclusion
$$
j_\nu\colon \Delta(\nu)\to\Delta(\alpha\uparrow\nu).
$$
We denote by $\ol j_\nu\colon \rDelta(\nu)\to\rDelta(\alpha\uparrow\nu)$ the homomorphism $j_\nu^{-}$ (which coincides with $j_\nu^{\res}$, as Verma modules are already positively restricted). Note that $\nu\not\le\alpha\uparrow\nu-n\delta$ for any $n>0$, hence $\ol j_\nu$ is non-zero.
Let $\CJ\subset\hfhd$ be open and bounded and suppose that $\gamma$ and $\alpha\uparrow\gamma$ are contained in $\CJ$. We also fix a surjection
$$
\pi_\nu\colon P^\CJ(\nu)\to\Delta(\nu)
$$
and an inclusion
$$
i_{\alpha\uparrow\nu}\colon \Delta(\alpha\uparrow\nu)\to P(\nu)^+.
$$
In particular, we have a short exact sequence
$$
0\to\Delta(\alpha\uparrow\nu)\stackrel{i_{\alpha\uparrow\nu}}\longrightarrow P(\nu)^+\stackrel{\pi_\nu^+}\longrightarrow\Delta(\nu)\to 0.
$$
As the action of $\CZ^-$ on Verma modules is free (see \cite{FG06} and \cite[Theorem 9.5.3]{Fre07}), this induces, after applying the functor $(\cdot)^-$, a short exact sequence
$$
0\to\rDelta(\alpha\uparrow\nu)\stackrel{\ol i_{\alpha\uparrow\nu}}\longrightarrow \rP(\nu)\stackrel{\ol \pi_\nu}\longrightarrow\rDelta(\nu)\to 0.
$$

Now we can find, by projectivity, a homomorphism 
$$
a_\gamma\colon P^\CJ(\gamma)\to P^\CJ(\alpha\uparrow\gamma)
$$ 
such that the diagram 

\centerline{
\xymatrix{
P^\CJ(\gamma)\ar[rr]^{a_{\gamma}}\ar[d]_{\pi_\gamma}&& P^\CJ(\alpha\uparrow\gamma)\ar[d]^{\pi_{\alpha\uparrow\gamma}}\\
\Delta(\gamma)\ar[rr]^{j_{\gamma}}&&\Delta(\alpha\uparrow\gamma).
}
}
\noindent
commutes. Applying the functor $(\cdot)^+$ yields a commuting diagram

\centerline{
\xymatrix{
P(\gamma)^+\ar[rr]^{ a^+_{\gamma}}\ar[d]_{\pi^+_\gamma}&& P(\alpha\uparrow\gamma)^+\ar[d]^{\pi^+_{\alpha\uparrow\gamma}}\\
\Delta(\gamma)\ar[rr]^{ j_{\gamma}}&&\Delta(\alpha\uparrow\gamma).
}
}
\noindent
The following is the crucial technical result of this paper. 
\begin{lemma}\label{lemma-main3} The composition $\Delta(\alpha\uparrow\gamma)\stackrel{i_{\alpha\uparrow\gamma}}\longrightarrow  P(\gamma)^+\stackrel{a^+_\gamma}\to P(\alpha\uparrow\gamma)^+$ is non-zero.
\end{lemma}

\begin{proof} Suppose the composition were zero. Then we could factor the map $a^+_\gamma$ over a homomorphism $\Delta(\gamma)\to P(\alpha\uparrow\gamma)^+$. By Proposition  \ref{prop-parresproj}, $P(\alpha\uparrow\gamma)^+\cong Z(\alpha\uparrow\gamma, \alpha\uparrow^2\gamma)$. 

Note that for any $\alpha$-subgeneric $\nu$, the weights $\nu$ and $\alpha\uparrow\nu$ are neighbouring. Moreover, with $\nu$ also $-2\rho-\nu$ is $\alpha$-subgeneric. Finally, in \cite{AF08} it is shown that for subgeneric $\nu$ we have
$$
\dim_\DC\Hom_{\hCO}(\Delta(\nu-n\delta),\Delta(\nu))=[ \Delta(\nu):L(\nu-n\delta)]
$$
for all $n\in\DZ$. Finally, $\alpha\uparrow^2\nu=\nu+n\delta$ for some $n>0$. This allows us to  apply Lemma \ref{lemma-main2} and we conclude  that the map $a^+_\gamma$  now would factor over a homomorphism $\Delta(\gamma)\to\Delta(\alpha\uparrow^2\gamma)\to P(\alpha\uparrow\gamma)^+$. But this contradicts its construction. 
\end{proof}

Now apply the restriction functor $(\cdot)^{\res}$ to $a_\gamma$. We obtain a homomorphism $\ol a_\gamma\colon \rP(\gamma)\to \rP(\alpha\uparrow\gamma)$ such that the diagram

\centerline{
\xymatrix{
\rP(\gamma)\ar[rr]^{ \ol a_{\gamma}}\ar[d]_{\ol\pi_\gamma}&& \rP(\alpha\uparrow\gamma)\ar[d]^{\ol\pi_{\alpha\uparrow\gamma}}\\
\rDelta(\gamma)\ar[rr]^{\ol j_{\gamma}}&&\rDelta(\alpha\uparrow\gamma).
}
}
\noindent
commutes. From Lemma \ref{lemma-main3} (and some weight considerations) we conclude:
\begin{lemma}\label{lemma-main4} The composition
$$
\rDelta(\alpha\uparrow\gamma)\stackrel{\ol i_{\alpha\uparrow\gamma}}\longrightarrow \rP(\gamma)\stackrel{\ol a_\gamma}\longrightarrow \rP(\alpha\uparrow\gamma)
$$
is non-zero.
\end{lemma}
 In particular, $\ol a_\gamma$ is non-zero, hence  a generator of $\Hom_{\rCO}(\rP(\gamma),\rP(\alpha\uparrow\gamma))$.

Let $b_\gamma\colon\rP(\gamma)\to\rP(\alpha\downarrow\gamma)$ be the following composition:
$$
b_\gamma\colon \rP(\gamma)\stackrel{\ol\pi_\gamma}\longrightarrow\rDelta(\gamma)\stackrel{\ol i_\gamma}\longrightarrow\rP(\alpha\downarrow\gamma),
$$
This composition is clearly non-zero, hence $b_\gamma$ is a basis of $\Hom_{\rCO}(\rP(\gamma), \rP(\alpha\downarrow\gamma))$.

Finally, let $\ol n_\gamma\colon \rP(\gamma)\to\rP(\gamma)$ be the composition
$$
\ol n_\gamma\colon \rP(\gamma)\stackrel{\ol\pi_\gamma}\longrightarrow\rDelta(\gamma)\stackrel{\ol j_\gamma}\longrightarrow\rDelta(\alpha\uparrow\gamma)\stackrel{\ol i_{\alpha\uparrow\gamma}}\longrightarrow \rP(\gamma).
$$
Again, this is non-zero and obviously not invertible (we even have $\ol n_\gamma^2=0$), so $\{n_\gamma,\id\}$ is a basis of $\End_{\rCO}(\rP(\gamma))$.

We have now exhibited a basis for any non-zero space $\Hom_{\rCO}(\rP(\gamma),\rP(\mu))$. The following theorem describes all possible (non-trivial) compositions, hence gives a full description of the subgeneric endomorphism algebra $\End_\rCO(\bigoplus_{\gamma\in\Gamma}\rP(\gamma))$, where $\Gamma$ is the $\rsim$-equivalence class of $\gamma$.

\begin{theorem} Let $\gamma\in\hfhd_{\crit}$ be $\alpha$-subgeneric. Then we have the following relations:
\begin{enumerate} 
\item $b_{\alpha\uparrow\gamma}\circ a_\gamma$ and $a_{\alpha\downarrow\gamma}\circ b_\gamma$ are non-zero scalar multiples of $n_\gamma$. 
\item $a_{\alpha\uparrow\gamma}\circ a_\gamma=0$ and $b_{\alpha\downarrow\gamma}\circ b_\gamma=0$.
\item $n_{\alpha\uparrow\gamma}\circ a_\gamma=0$ and $n_{\alpha\downarrow\gamma}\circ b_\gamma=0$. 
\item $n_\gamma\circ n_\gamma=0$.
\end{enumerate}
\end{theorem}

\begin{proof} Note that (2) is obvious, as the  homomorphism spaces in  question vanish. Then (3) and (4) follow immediately from (1) and (2). So we are left to prove (1). Note that both compositions are clearly not automorphisms of $\rP(\gamma)$, so we only have to prove that they are non-zero.
From the construction it immediately follows that $b_{\alpha\uparrow\gamma}\circ a_\gamma\ne 0$. That $a_{\alpha\uparrow\gamma}\circ b_\gamma$ is non-zero follows from Lemma \ref{lemma-main4}.
\end{proof}

Hence we see that the endomorphism algebra of $\bigoplus_{\gamma\in\Gamma} \rP(\gamma)$ is given by the following infinite quiver

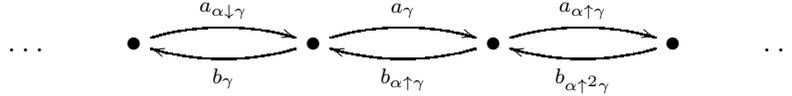
\centerline{
\xymatrix{
\dots&\bullet\ar@/^/[rr]^{a_{\alpha\downarrow\gamma}}&& \bullet \ar@/^/[ll]^{b_\gamma}\ar@/^/[rr]^{a_{\gamma}}&& \bullet \ar@/^/[ll]^{b_{\alpha\uparrow\gamma}}\ar@/^/[rr]^{a_{\alpha\uparrow\gamma}}&& \bullet \ar@/^/[ll]^{b_{\alpha\uparrow^2\gamma}}& \dots
}
}
 \noindent
with relations $a_{\alpha\downarrow\gamma}\circ b_{\gamma}=b_{\alpha\uparrow\gamma}\circ a_{\gamma}$ and $a_{\alpha\uparrow\gamma}\circ a_\gamma=0$ and $b_{\gamma}\circ b_{\alpha\uparrow\gamma}=0$ for all $\gamma\in\Gamma$.

\end{document}